\documentclass[12pt,reqno]{amsart}

\usepackage{xcolor}
\usepackage{comment}
\usepackage{amsmath}
\usepackage{amssymb}
\usepackage{amsthm}
\usepackage[latin1]{inputenc}
\usepackage{eurosym}
\usepackage{graphicx}
\usepackage{epsfig}
\usepackage{hyperref}
\usepackage{dsfont}
\usepackage[nocompress, space]{cite}
\usepackage{caption}
\usepackage{subcaption}
\usepackage[displaymath,mathlines]{lineno}
\usepackage{enumerate}
\usepackage{cancel}

\allowdisplaybreaks

\usepackage{ifthen}

\newcommand{\argmax}{\operatorname{argmax}}

\def\leq{\leqslant}
\def\geq{\geqslant}

\numberwithin{equation}{section}

\newtheoremstyle{thmlemcorr}{10pt}{10pt}{\itshape}{}{\bfseries}{.}{10pt}{{\thmname{#1}\thmnumber{
#2}\thmnote{ (#3)}}}
\newtheoremstyle{thmlemcorr*}{10pt}{10pt}{\itshape}{}{\bfseries}{.}\newline{{\thmname{#1}\thmnumber{
#2}\thmnote{ (#3)}}}
\newtheoremstyle{defi}{10pt}{10pt}{\itshape}{}{\bfseries}{.}{10pt}{{\thmname{#1}\thmnumber{
#2}\thmnote{ (#3)}}}
\newtheoremstyle{remexample}{10pt}{10pt}{}{}{\bfseries}{.}{10pt}{{\thmname{#1}\thmnumber{
#2}\thmnote{ (#3)}}}
\newtheoremstyle{ass}{10pt}{10pt}{}{}{\bfseries}{.}{10pt}{{\thmname{#1}\thmnumber{
A#2}\thmnote{ (#3)}}}

\theoremstyle{thmlemcorr}
\newtheorem{theorem}{Theorem}
\numberwithin{theorem}{section}

\newtheorem{proposition}[theorem]{Proposition}

\theoremstyle{thmlemcorr*}
\newtheorem{theorem*}{Theorem}
\newtheorem{lemma*}[theorem]{Lemma}
\newtheorem{corollary*}[theorem]{Corollary}
\newtheorem{proposition*}[theorem]{Proposition}
\newtheorem{problem*}[theorem]{Problem}
\newtheorem{conjecture*}[theorem]{Conjecture}

\theoremstyle{defi}

\theoremstyle{remexample}

\theoremstyle{plain}

\theoremstyle{ass}

\title[Forced mean curvature flow]{A convergence rate of periodic homogenization for forced mean curvature flow of graphs in the laminar setting}

\begin{document}


\author{Jiwoong Jang}
\address[Jiwoong Jang]
{
	Department of Mathematics, 
	University of Wisconsin Madison, Van Vleck hall, 480 Lincoln drive, Madison, WI 53706, USA}
\email{jjang57@wisc.edu}


\thanks{
	The work of JJ was partially supported by NSF CAREER grant DMS-1843320.
}


\begin{abstract}
In this paper, we obtain the rate $O(\varepsilon^{1/2})$ of convergence in periodic homogenization of forced graphical mean curvature flows in the laminated setting. We also discuss with an example that a faster rate cannot be obtained by utilizing Lipscthiz estimates.
\end{abstract}

\maketitle


\section{Introduction} \label{sec:introduction}
In this paper, we are interested in the quantitative understanding of convergence of graphical hypersurfaces $\Gamma^{\varepsilon}(t)(\subseteq\mathbb{R}^{n+1})$ to $\Gamma(t)$ as $\varepsilon\to0$ in a laminated environment, where the hypersurfaces $\Gamma^{\varepsilon}(t)$ evolve by the normal velocity
$$
V=\varepsilon\kappa+c\left(\frac{x}{\varepsilon}\right).
$$
Here, $\kappa$ is the mean curvature of the hypersurface, and $c$ is a given force depending on the spatial variable periodically. Fixing axes of $\mathbb{R}^{n+1}$, we write $c=c(x)$ with $x=(x_1,\cdots,x_n)$. The media is laminated so that $c$ is independent of $x_{n+1}$.

If $\Gamma^{\varepsilon}(t)$ has a height function $u^{\varepsilon}(\cdot,t)$ so that $\Gamma^{\varepsilon}(t)=\{(x,u^{\varepsilon}(x,t)):x\in\mathbb{R}^n\}$, then the evolution of hypersurfaces $\Gamma^{\varepsilon}(t)$, with an initial graph $\{(x,u_0(x)):x\in\mathbb{R}^n\}$, is described by the equation
\begin{equation}\label{eq:epsilon}
\begin{cases}
u^{\varepsilon}_t+F\left(\varepsilon D^2u^{\varepsilon},Du^{\varepsilon},\frac{x}{\varepsilon}\right)=0 \quad &\text{ in } \mathbb{R}^n\times(0,\infty),\\
u^{\varepsilon}(x,0)=u_0(x) \quad &\text{ on } \mathbb{R}^n
\end{cases}
\end{equation}
for $\varepsilon\in(0,1]$, where $F=F(X,p,y)$ is the mean curvature operator with a forcing term of graphs
$$
F(X,p,y)=-\textrm{tr}\left\{\left(I_n-\frac{p\otimes p}{1+|p|^2}\right)X\right\}-c\left(y\right)\sqrt{1+|p|^2},
$$
for $(X,p,y)\in S^n\times\mathbb{R}^n\times\mathbb{R}^n$, $n\geq1$. The precise meaning of notations will be introduced later.

Throughout this paper, we impose the following assumptions on the forcing term $c$;
\begin{align*}
&\textrm{(A1)}\ c\in C^{2}(\mathbb{R}^n);\hspace{20mm}\\
&\textrm{(A2)}\ c=c(y)\ \textrm{is $\mathbb{Z}^n$-periodic in }y\in\mathbb{R}^n\textrm{, i.e., $c(y+k)=c(y)$ for $k\in\mathbb{Z}^n,\ y\in\mathbb{R}^n$};\\
&\textrm{(A3)}\ c(y)^2-(n-1)|Dc(y)|>\delta\ \ \ \textrm{for all }y\in\mathbb{R}^n,\ \textrm{for some }\delta>0.
\end{align*}
We also assume that $u_0\in\mathrm{Lip}(\mathbb{R}^n)$.

Under the assumptions (A1)--(A3), it is known (see \cite{E,LS} for instance) that $u^{\varepsilon}$ converges locally uniformly to $u$ as $\varepsilon\to0^{+}$ on $\mathbb{R}^n\times[0,\infty)$, which is a viscosity solution to the effective equation
\begin{equation}\label{eq:effec}
\begin{cases}
u_t+\overline{F}(Du)=0 \quad &\text{ in } \mathbb{R}^n\times(0,\infty),\\
u(x,0)=u_0(x) \quad &\text{ on } \mathbb{R}^n.
\end{cases}
\end{equation}
Here, $\overline{F}(p)$ is the unique real number such that the cell problem
$$
F(D^2v,p+Dv,y)=\overline{F}(p)\hspace{1cm}\textrm{on }\mathbb{R}^n.
$$
admits a $\mathbb{Z}^n$-periodic solution $v\in C^{2,\alpha}(\mathbb{R}^n)$ for some $\alpha\in(0,1)$. We refer to \cite{E,LPV} for the definition of $\overline{F}(p)$, or that of the effective Hamiltonian $\overline{H}(p)$.

The main goal of this paper is to obtain a rate of convergence of $u^{\varepsilon}$ to $u$ as $\varepsilon\to0^+$ by proving (i) that $\|u^{\varepsilon}-u\|_{L^{\infty}(\mathbb{R}^n\times[0,T])}$ is $O(\varepsilon^{1/2})$ for any given $T>0$, and (ii) that $|u^{\varepsilon}(x_0,t_0)-u(x_0,t_0)|$ is $\Omega(\varepsilon^{1/2})$, i.e., $|u^{\varepsilon}(x_0,t_0)-u(x_0,t_0)|\geq C\varepsilon^{1/2}$ for some $C>0,$  $(x_0,t_0)\in\mathbb{R}^n\times(0,\infty)$ in certain cases.

\subsection{Literature and main results}\label{subsec:literature}

\medskip

Periodic homogenization of geometric motions has been studied only recently. In \cite{LS}, Lipschitz continuous correctors were found under the assumption (A3) in the periodic setting. When the gradient of the force $c$ is large and $n\geq3$, it is shown in \cite{CM} by an example in the laminar setting that homogenization may not occur. It is also shown in \cite{CM} that homogenization always takes place when $n=2$ for the level-set fronts as long as the force is positive, whose argument is 2-d arguments, showing the front is trapped in two parallel translations of an initial front in a bounded distance. Without any sign condition on $c$, Lipschitz continuous correctors were found in \cite{DKY} under the condition that $c\in C^2(\mathbb{T}^n)$ and that $\|c\|_{C^2(\mathbb{T}^n)}$ is small enough, whose part of the proof is based on \cite{CL}. A further analysis on asymptotic speeds is given in \cite{GK}. For more related works, we refer to \cite{CB, CLS, CN}. See also the recent works \cite{GLXY,MMTXY} on the curvature $G$-equation. To the best of our knowledge, quantitative homogenization of geometric motions in the periodic environment has not been treated.

Quantitative homogenization for Hamilton-Jacobi equations in the periodic setting has received a lot of attention. The rate $O(\varepsilon^{1/3})$ was obtained for first-order equations in \cite{CI}. For convex first-order Hamilton-Jacobi equations, the optimal rate of convergence $O(\varepsilon)$ was obtained very recently in \cite{TY}. We refer to \cite{HJ,MTY,TY,T} and the references therein for earlier progress in this direction.

In this paper, we obtain the rate $O(\varepsilon^{1/2})$ for periodic homogenization of forced mean curvature flow of graphs. We follow the framework of \cite{CI}, and we utilize the additional fact that there is a regular selection of correctors (see Proposition \ref{prop:cellproblem}). Based on this observation, we derive the improved rate $O(\varepsilon^{1/2})$. Also, we list an example that shows that we cannot expect a faster rate than $O(\varepsilon^{1/2})$ if we expect only the Lipschitz continuity of solutions and a regular selection of correctors. In the study of Hamilton-Jacobi equations, this improvement of rates is noted in \cite[Theorem 4.40]{Tran} and used to obtain the optimal rate of periodic homogenization of viscous Hamilton-Jacobi equations in \cite{QTY}. Our work is closely related to \cite{QTY}, \cite[Theorem 4.40]{Tran}.

\medskip

We now give the precise statements of our main results. The rate $O(\varepsilon^{1/2})$ is obtained in the following theorem.

\begin{theorem}\label{thm:provingrate}
Assume {\rm(A1)-(A3)}, and let $u_0$ be a globally Lipschitz function on $\mathbb{R}^n$ with $\|Du_0\|_{L^{\infty}(\mathbb{R}^n)}\leq N_0<+\infty$. For $\varepsilon\in(0,1]$, let $u^{\varepsilon}$ be the unique classical solution to \eqref{eq:epsilon}, and let $u$ be the unique viscosity solution to \eqref{eq:effec}. Fix $T>0$. Then, there exists a constant $C>0$ depending only on $n,\|c\|_{C^2(\mathbb{R}^n)},N_0,\delta$ such that
\begin{align*}
\|u^{\varepsilon}-u\|_{L^{\infty}(\mathbb{R}^n\times[0,T])}\leq C(1+T)\varepsilon^{1/2}. 
\end{align*}

\end{theorem}

The next theorem shows that in the absence of a forcing term, i.e., $c\equiv0$, one obtains the rate $\Omega(\varepsilon^{1/2})$. Also, one can check that the rate is $O(\varepsilon^{1/2})$ for general Lipschitz continuous, positively 1-homogeneous initial data when $c\equiv0$.

\begin{theorem}\label{thm:example}
Let $c\equiv0$, and let $u_0(x)=|x|$. For $\varepsilon\in(0,1]$, let $u^{\varepsilon}$ be the unique classical solution to \eqref{eq:epsilon}, and let $u$ be the unique viscosity solution to \eqref{eq:effec}. Then, there exists an absolute constant $C>0$ such that
$$
|u^{\varepsilon}(0,1)-u(0,1)|\geq C\varepsilon^{1/2}.
$$
\end{theorem}

\subsection*{Organization of the paper}
In Section \ref{sec:wellposedness}, we state propositions about the well-posedness of \eqref{eq:epsilon}. In Section \ref{sec:provingrate}, we simplify the settings of the problem by using a priori estimates, and give a proof of Theorem \ref{thm:provingrate}. In Section \ref{sec:example}, we obtain the optimality of the rate in Theorem \ref{thm:provingrate} by proving Theorem \ref{thm:example}.

\subsection*{Notations}
The set of all $n$ by $n$ matrices is denoted by $S^n$. The matrix $I_n$ denotes the $n$ by $n$ identity matrix, and $p\otimes p$ is the matrix $\left(p^ip^j\right)_{i,j=1}$ for $p=(p^1,\cdots,p^n)^{\textrm{t}}\in\mathbb{R}^n$.

In the subsequent sections, $\langle x\rangle$ denotes the number $(1+|x|^2)^{1/2}$ for $x\in\mathbb{R}^n$. Note that $D\langle x\rangle=\frac{x}{\langle x\rangle}$, and $D^2\langle x\rangle=\frac{1}{\langle x\rangle}\left(I_n-\frac{x}{\langle x\rangle}\otimes\frac{x}{\langle x\rangle}\right)$. We also let, for $p\in\mathbb{R}^n$, $a(p)$ denote the matrix $I_n-\frac{p}{\langle p\rangle}\otimes\frac{p}{\langle p\rangle}$. For a nonzero vector $p$ in $\mathbb{R}^n$ (or in $\mathbb{R}^{n+1}$), we let $\widehat{p}=\frac{p}{|p|}$. For a square matrix $\alpha$, we let $\|\alpha\|=\sqrt{\textrm{tr}\{\alpha^t\alpha\}}$, where $\textrm{tr}\{\cdot\}$ is the trace of a given argument square matrix. Numbers $C,M>0$ denotes constants that may vary line by line, and their dependency on parameters will be specified in arguments.

For $p\in\mathbb{R}^n$, we let $a^{ij}(p)$ be the $(i,j)$-entry of the matrix $a(p)$. We define $D_pa(p)\odot q$ to be the matrix $\left(\sum_{k=1}^n\left(\frac{\partial}{\partial p^k}a^{ij}(p)\right)q^k\right)_{i,j=1,\cdots,n}$ for $q=(q^1,\cdots,q^n)\in\mathbb{R}^n$.


\section{Proof of Theorem \ref{thm:provingrate}}\label{sec:provingrate}
\subsection{Well-posedness of \eqref{eq:epsilon}}\label{sec:wellposedness}
We consider the forced mean curvature flow of graphs
\begin{equation}\label{eq:w}
\begin{cases}
w_t=\textrm{tr}\left\{a(Dw)D^2w\right\}+c(x)\sqrt{1+|Dw|^2} \quad &\text{ in } \mathbb{R}^n\times(0,T),\\
w(x,0)=w_0(x) \quad &\text{ on } \mathbb{R}^n.
\end{cases}
\end{equation}
We note that by change of variables, namely $u^{\varepsilon}(x,t)=\varepsilon w\left(\frac{x}{\varepsilon},\frac{t}{\varepsilon}\right)$, or $w(x,t)=\frac{1}{\varepsilon}u^{\varepsilon}(\varepsilon x,\varepsilon t)$, we can go back and forth between \eqref{eq:epsilon} and \eqref{eq:w} (when $T=+\infty$), with the change $w_0(x)=\frac{1}{\varepsilon}u_0(\varepsilon x),\ u_0(x)=\varepsilon w_0(\frac{1}{\varepsilon} x)$. We also note that Lipschitz constants on initial data are preserved through this change of variables.

We state the well-posedness of \eqref{eq:w}, which ensures that of \eqref{eq:epsilon}.

\begin{theorem}\label{thm:wellposedness}
Assume \rm(A1) and \rm(A3). Let $w_0$ be a globally Lipschitz function on $\mathbb{R}^n$ with $\|Dw_0\|_{L^{\infty}(\mathbb{R}^n)}$ $\leq N_0<+\infty$. Then, \eqref{eq:w} has a unique classical solution for all time ($T=+\infty$), and moreover, there exists $M=M(n,\|c\|_{C^2(\mathbb{R}^n)},N_0,\delta)>0$ such that
$$
\|Dw\|_{L^{\infty}(\mathbb{R}^n\times[0,\infty))}\leq M,
$$
and
\begin{align}
\|w_t(\cdot,t)\|_{L^{\infty}(\mathbb{R}^n)}\leq M\left(\frac{1}{\min\{\sqrt{t},1\}}+1\right)\quad\text{for all }t>0.\label{line:timederivativeofw}  
\end{align}
Here, $\delta>0$ is the number appearing in the condition \rm(A3).
\end{theorem}
We outline a sketch of this theorem in Appendix \ref{appendix}. The references for the theorem we refer to are \cite{EH} (when $c\equiv0$) and \cite[Appendix A]{DKY} (with a forcing term $c$).

\subsection{Settings and simplifications}\label{subsec:simplification}
We assume the conditions (A1)-(A3) in the rest of this section. Let $u_0$ be a globally Lipschitz function on $\mathbb{R}^n$ with $\|Du_0\|_{L^{\infty}(\mathbb{R}^n)}\leq N_0<+\infty$, we consider \eqref{eq:epsilon}. Then, $w(x,t)=\frac{1}{\varepsilon}u^{\varepsilon}(\varepsilon x,\varepsilon t)$ solves \eqref{eq:w} with $w_0(x)=\frac{1}{\varepsilon}u_0(\varepsilon x)$ with the same Lipschitz constant $N_0$. Through this change of variables, we see that there exists  $M=M(n,\|c\|_{C^2(\mathbb{R}^n)},N_0,\delta)>0$ such that
\begin{align}
\|u_t^{\varepsilon}\|_{L^{\infty}(\mathbb{R}^n\times[\varepsilon,\infty))}+\|Du^{\varepsilon}\|_{L^{\infty}(\mathbb{R}^n\times[0,\infty))}\leq M.\label{line:gradientbound}  
\end{align}
Also, by \eqref{line:timederivativeofw}, we have
\begin{align}
\|u_t^{\varepsilon}(\cdot,t)\|_{L^{\infty}(\mathbb{R}^n)}\leq M\left(\left\|D^2w\left(\cdot,\frac{t}{\varepsilon}\right)\right\|_{L^{\infty}(\mathbb{R}^n)}+1\right)\leq M\left(\max\left\{\sqrt{\frac{\varepsilon}{t}},1\right\}+1\right)\label{line:timederivative}
\end{align}
for $t>0$. Combining \eqref{line:timederivative} with \eqref{line:gradientbound}, we see that for each compact set $K\subseteq\mathbb{R}^n\times[0,\infty)$, there exists $M=M(n,\|c\|_{C^2(\mathbb{R}^n)},N_0,\delta,K)>0$ such that $\|u^{\varepsilon}\|_{L^{\infty}(K)}\leq M$ for each $\varepsilon\in(0,1]$ by integration. By the Arzel\`a-Ascoli Theorem, $u^{\varepsilon}$ converges to $u$ locally uniformly on $\mathbb{R}^n\times[0,\infty)$ as $\varepsilon\to0$, and $u$ solves \eqref{eq:effec} (see \cite{E}), and satisfies
\begin{align}
\|u_t\|_{L^{\infty}(\mathbb{R}^n\times[0,\infty))}+\|Du\|_{L^{\infty}(\mathbb{R}^n\times[0,\infty))}\leq M.\label{line:gradientboundofu}  
\end{align}
Therefore, changing the values of $F(X,p,x)$ for $|p|> M$ does not affect the equations \eqref{eq:epsilon} and \eqref{eq:effec}.

Let $\xi\in C^{\infty}(\mathbb{R}^n,[0,1])$ be a cut-off function such that

\begin{equation*}
\xi(r)=
\begin{cases}
1 \quad &\text{ for } r\leq\sqrt{1+M^2}+1,\\
0 \quad &\text{ for } r\geq\sqrt{1+M^2}+2.
\end{cases}
\end{equation*}
Let
$$
\Tilde{F}(X,p,y)=-\textrm{tr}\left\{\left(I_n-\frac{p\otimes p}{1+|p|^2}\right)X\right\}-\Tilde{c}(y,p)\sqrt{1+|p|^2}
$$
for $(X,p,x)\in S^n\times\mathbb{R}^n\times\mathbb{R}^n$, where
$$
\Tilde{c}(y,p)=\xi\left(\sqrt{1+|p|^2}\right)c(y)+\left(1-\xi\left(\sqrt{1+|p|^2}\right)\right)c_0.
$$
Here, $c_0=\sup_{y\in\mathbb{R}^n}(c(y))$ if $c>0$, and $c_0=\inf_{y\in\mathbb{R}^n}(c(y))$ if $c<0$. Note that $c=c(y)$ is either always positive or always negative due to the assumption (A3). From this choice of the constant $c_0$, $\Tilde{c}=\Tilde{c}(y,p)$ satisfies
\begin{align*}
&\textrm{(A1)'}\ \Tilde{c}(y,p)\textrm{ is $C^2$ in $y\in\mathbb{R}^n$ and $C^{\infty}$ in $p\in\mathbb{R}^n$};\hspace{20mm}\\
&\textrm{(A2)'}\ \Tilde{c}(y+k,p)=\Tilde{c}(y,p)\textrm{ for all $y,p\in\mathbb{R}^n,\ k\in\mathbb{Z}^n$};\\
&\textrm{(A3)'}\ \Tilde{c}(y,p)^2-(n-1)|D_y\Tilde{c}(y,p)|>\delta\ \textrm{for }y,p\in\mathbb{R}^n,\ \textrm{for the same }\delta>0 \textrm{ in (A3)}.
\end{align*}
Also, it holds that $\Tilde{c}(y,p)\in[\min_{\mathbb{R}^n}(c),\max_{\mathbb{R}^n}(c)]$ for all $(y,p)\in\mathbb{R}^n\times\mathbb{R}^n$. We note that $u^{\varepsilon}$ solves \eqref{eq:epsilon} with $\Tilde{F}$ in place of $F$ as expected.

Since the modified force $\Tilde{c}=\Tilde{c}(y,p)$ satisfies the assumption (A3)', we have the following proposition (see \cite[Proposition 3.1]{LS}).

\begin{proposition}\label{prop:cellproblem}
For each $p\in\mathbb{R}^n$, there exists a unique real number, denoted by $\overline{\Tilde{F}}(p)$, such that the cell problem
\begin{equation}\label{eq:cellpb}
\begin{cases}
\Tilde{F}(D^2\Tilde{v},p+D\Tilde{v},y)=\overline{\Tilde{F}}(p) \quad &\text{ on } \mathbb{R}^n,\\
\Tilde{v}(0,p)=0. \quad &
\end{cases}
\end{equation}
has a unique $\mathbb{Z}^n$-periodic solution $\Tilde{v}=\Tilde{v}(\cdot,p)\in C^{2}(\mathbb{R}^n)$. Moreover, for each $y\in\mathbb{R}^n$, the map $p\mapsto \Tilde{v}(y,p)$ is well-defined and is a $C^2$ map from $\mathbb{R}^n$ to $\mathbb{R}$. Also, the map $p\mapsto \overline{\Tilde{F}}(p)$ is a $C^2$ map from $\mathbb{R}^n$ to $\mathbb{R}$, and $\overline{\Tilde{F}}(p)\in[\min_{\mathbb{R}^n}(-\Tilde{c}(\cdot,p))\sqrt{1+|p|^2},\max_{\mathbb{R}^n}(-\Tilde{c}(\cdot,p))\sqrt{1+|p|^2}]$ for all $p\in\mathbb{R}^n$.
\end{proposition}

Before we move on to the proof of Theorem \ref{thm:provingrate}, we explain the additional property coming from the modified operator $\Tilde{F}$. For $|p|\geq M+10$, $\Tilde{c}(y,p)=c_0$, and thus, $\Tilde{v}(\cdot,p)\equiv0$. Therefore, there exists a constant $C=C(n,\|c\|_{C^2(\mathbb{R}^n)},N_0,\delta)>0$ such that
\begin{align}
\sup_{y,p\in\mathbb{R}^n}|\Tilde{v}(y,p)|+\sup_{y,p\in\mathbb{R}^n}|D_{(y,p)}\Tilde{v}(y,p)|+\sup_{y,p\in\mathbb{R}^n}\|D^2\Tilde{v}(y,p)\|\leq C.\label{ineq:globalbound}
\end{align}
Here, the gradient $D_{(y,p)}\Tilde{v}(y,p)$ is with respect to the variable $(y,p)\in\mathbb{R}^n\times\mathbb{R}^n$, and the Hessian $D^2\Tilde{v}(y,p)$ is with respect to the variable $y\in\mathbb{R}^n$.


\subsection{Proof}\label{subsec:proofprovingrate}
We prove that for a fixed $T>0$,
\begin{align}
\|u^{\varepsilon}-u\|_{L^{\infty}(\mathbb{R}^n\times[0,T])}\leq C(1+T)\varepsilon^{1/2}\label{goal:sec3}
\end{align}
for some constant $C=C(n,\|c\|_{C^2(\mathbb{R}^n)},N_0,\delta)>0$. From now on, we use $F$ and $v$ to denote $\Tilde{F}$ and $\Tilde{v}$, respectively, by abuse of notations.

\begin{proof}[Proof of Theorem \ref{thm:provingrate}]
We first show that
$$
u^{\varepsilon}(x,t)-u(x,t)\leq C(1+T)\varepsilon^{1/2}
$$
for $(x,t)\in\mathbb{R}^n\times[0,T]$. We set the auxiliary function
\begin{align*}
\Phi(x,y,z,t,s):=u^{\varepsilon}(x,t)&-u(y,s)-\varepsilon v\left(\frac{x}{\varepsilon},\frac{z-y}{\varepsilon^{1/2}}\right)\\
&-\frac{|x-y|^2+|t-s|^2}{2\varepsilon^{1/2}}-\frac{|x-z|^2}{2\varepsilon^{1/2}}-K(t+s)-\gamma\langle x\rangle,
\end{align*}
where $K,\gamma>0$ are numbers that will be chosen later. Then, the global maximum of $\Phi$ on $\mathbb{R}^{3n}\times[0,T]^2$ is attained at a certain point $(\hat{x},\hat{y},\hat{z},\hat{t},\hat{s})\in\mathbb{R}^{3n}\times[0,T]^2$.

From $\Phi(\hat{x},\hat{y},\hat{z},\hat{t},\hat{s})\geq\Phi(\hat{x},\hat{y},\hat{x},\hat{t},\hat{s})$ with \eqref{ineq:globalbound}, we have
$$
\frac{|\hat{x}-\hat{z}|^2}{2\varepsilon^{1/2}}\leq\varepsilon\left(v\left(\frac{\hat{x}}{\varepsilon},\frac{\hat{x}-\hat{y}}{\varepsilon^{1/2}}\right)-v\left(\frac{\hat{x}}{\varepsilon},\frac{\hat{z}-\hat{y}}{\varepsilon^{1/2}}\right)\right)\leq C\varepsilon^{1/2}|\hat{x}-\hat{z}|,
$$
which gives $|\hat{x}-\hat{z}|\leq C\varepsilon$. Similarly, from $\Phi(\hat{x},\hat{y},\hat{z},\hat{t},\hat{s})\geq\Phi(\hat{x},\hat{x},\hat{x},\hat{t},\hat{s})$, we get
\begin{align*}
\frac{|\hat{x}-\hat{y}|^2+|\hat{x}-\hat{z}|^2}{2\varepsilon^{1/2}}&\leq u(\hat{x},\hat{s})-u(\hat{y},\hat{s})+\varepsilon\left(v\left(\frac{\hat{x}}{\varepsilon},0\right)-v\left(\frac{\hat{x}}{\varepsilon},\frac{\hat{z}-\hat{y}}{\varepsilon^{1/2}}\right)\right)\\
&\leq C|\hat{x}-\hat{y}|+C\varepsilon^{1/2}|\hat{y}-\hat{z}|,
\end{align*}
which yields $|\hat{x}-\hat{y}|+|\hat{y}-\hat{z}|\leq C\varepsilon^{1/2}$. Lastly, $\Phi(\hat{x},\hat{y},\hat{z},\hat{t},\hat{s})\geq\Phi(\hat{x},\hat{y},\hat{z},\hat{t},\hat{t})$ and $\Phi(\hat{x},\hat{y},\hat{z},\hat{t},\hat{s})\geq\Phi(\hat{x},\hat{y},\hat{z},\hat{s},\hat{s})$ give $|\hat{t}-\hat{s}|\leq C\varepsilon^{1/2}$, which can be obtained by using \eqref{line:gradientbound}, \eqref{line:timederivative}, \eqref{line:gradientboundofu}.

Next, we show that the case $\hat{t},\hat{s}>0$ does not happen if $\gamma\leq\varepsilon^{1/2},\ K=K_1\varepsilon^{1/2}$, $K_1>0$ a constant sufficiently large. Assume first that $\hat{t},\hat{s}>0$. Then, $(x,t)\mapsto\Phi(x,\hat{y},\hat{z},t,\hat{s})$ attains a maximum at $(\hat{x},\hat{t})$, and thus, by the maximum principle,
\begin{align*}
Du^{\varepsilon}(\hat{x},\hat{t})&=Dv\left(\frac{\hat{x}}{\varepsilon},\frac{\hat{z}-\hat{y}}{\varepsilon^{1/2}}\right)+\frac{\hat{x}-\hat{y}}{\varepsilon^{1/2}}+\frac{\hat{x}-\hat{z}}{\varepsilon^{1/2}}+\gamma\frac{\hat{x}}{\langle\hat{x}\rangle},\\
u_t^{\varepsilon}(\hat{x},\hat{t})&\geq K+\frac{\hat{t}-\hat{s}}{\varepsilon^{1/2}},\\
D^2u^{\varepsilon}(\hat{x},\hat{t})&\leq\frac{1}{\varepsilon}D^2v\left(\frac{\hat{x}}{\varepsilon},\frac{\hat{z}-\hat{y}}{\varepsilon^{1/2}}\right)+\frac{2}{\varepsilon^{1/2}}I_n+\frac{\gamma}{\langle\hat{x}\rangle}\left(I_n-\frac{\hat{x}}{\langle\hat{x}\rangle}\otimes\frac{\hat{x}}{\langle\hat{x}\rangle}\right).
\end{align*}
From $u_{t}^{\varepsilon}(\hat{x},\hat{t})+F\left(\varepsilon D^2u^{\varepsilon}(\hat{x},\hat{t}),Du^{\varepsilon}(\hat{x},\hat{t}),\frac{\hat{x}}{\varepsilon}\right)=0$, we obtain
\begin{align}
0&\geq K+\frac{\hat{t}-\hat{s}}{\varepsilon^{1/2}}+F\left(D^2v\left(\frac{\hat{x}}{\varepsilon},\frac{\hat{z}-\hat{y}}{\varepsilon^{1/2}}\right)+2\varepsilon^{1/2}I_n\right.\nonumber\\
&\qquad\qquad\qquad\qquad\qquad\qquad\qquad\left.+\frac{\varepsilon\gamma}{\langle\hat{x}\rangle}\left(I_n-\frac{\hat{x}}{\langle\hat{x}\rangle}\otimes\frac{\hat{x}}{\langle\hat{x}\rangle}\right),Du^{\varepsilon}(\hat{x},\hat{t}),\frac{\hat{x}}{\varepsilon}\right)\nonumber\\
&=K+\frac{\hat{t}-\hat{s}}{\varepsilon^{1/2}}+F\left(D^2v\left(\frac{\hat{x}}{\varepsilon},\frac{\hat{z}-\hat{y}}{\varepsilon^{1/2}}\right),Du^{\varepsilon}(\hat{x},\hat{t}),\frac{\hat{x}}{\varepsilon}\right)\nonumber\\
&\hspace{2.8cm}-\textrm{tr}\left\{a\left(Du^{\varepsilon}(\hat{x},\hat{t})\right)\left(2\varepsilon^{1/2}I_n+\frac{\varepsilon\gamma}{\langle\hat{x}\rangle}\left(I_n-\frac{\hat{x}}{\langle\hat{x}\rangle}\otimes\frac{\hat{x}}{\langle\hat{x}\rangle}\right)\right)\right\}.\nonumber
\end{align}
Since $0\leq a(Du^{\varepsilon}(\hat{x},\hat{t}))\leq I_n$ and $0\leq 2\varepsilon^{1/2}I_n+\frac{\varepsilon\gamma}{\langle\hat{x}\rangle}\left(I_n-\frac{\hat{x}}{\langle\hat{x}\rangle}\otimes\frac{\hat{x}}{\langle\hat{x}\rangle}\right)\leq(2\varepsilon^{1/2}+\varepsilon\gamma)I_n$, we obtain
\begin{align}
0\geq K+\frac{\hat{t}-\hat{s}}{\varepsilon^{1/2}}+F\left(D^2v\left(\frac{\hat{x}}{\varepsilon},\frac{\hat{z}-\hat{y}}{\varepsilon^{1/2}}\right),Du^{\varepsilon}(\hat{x},\hat{t}),\frac{\hat{x}}{\varepsilon}\right)-n\varepsilon^{1/2}(2+\gamma\varepsilon^{1/2}).\label{subsolntest0}
\end{align}

Besides, $v=v\left(\cdot,\frac{\hat{z}-\hat{y}}{\varepsilon^{1/2}}\right)$ solves
$$
F\left(D^2v,\frac{\hat{z}-\hat{y}}{\varepsilon^{1/2}}+Dv,y\right)=\overline{F}\left(\frac{\hat{z}-\hat{y}}{\varepsilon^{1/2}}\right).
$$
Since $\left|Du^{\varepsilon}(\hat{x},\hat{t})-\left(\frac{\hat{z}-\hat{y}}{\varepsilon^{1/2}}+Dv\right)\right|=\left|\frac{2(\hat{x}-\hat{z})}{\varepsilon^{1/2}}+\gamma\frac{\hat{x}}{\langle \hat{x}\rangle}\right|\leq C(\varepsilon^{1/2}+\gamma)$,
\begin{align}
&\left|F\left(D^2v\left(\frac{\hat{x}}{\varepsilon},\frac{\hat{z}-\hat{y}}{\varepsilon^{1/2}}\right),Du^{\varepsilon}(\hat{x},\hat{t}),\frac{\hat{x}}{\varepsilon}\right)\right.\nonumber\\
&\hspace{3cm}\left.-F\left(D^2v\left(\frac{\hat{x}}{\varepsilon},\frac{\hat{z}-\hat{y}}{\varepsilon^{1/2}}\right),\frac{\hat{z}-\hat{y}}{\varepsilon^{1/2}}+Dv\left(\frac{\hat{x}}{\varepsilon},\frac{\hat{z}-\hat{y}}{\varepsilon^{1/2}}\right),\frac{\hat{x}}{\varepsilon}\right)\right|\nonumber\\
&\leq C(\varepsilon^{1/2}+\gamma)\left(\left\|D^2v\left(\frac{\hat{x}}{\varepsilon},\frac{\hat{z}-\hat{y}}{\varepsilon^{1/2}}\right)\right\|+\max_{y\in\mathbb{R}^n}|c(y)|\right)\nonumber\\
&\leq C(\varepsilon^{1/2}+\gamma)\label{subsolntest1}.
\end{align}
Here, we are using the property of $a$ that for $p,q\in\mathbb{R}^n$,
\begin{align*}
\|a(p)-a(q)\|
\leq C|p-q|,
\end{align*}
and using Cauchy-Schwarz's inequality $\left|\textrm{tr}\{\alpha\beta^t\}\right|\leq\|\alpha\|\|\beta\|$ for two square matrices $\alpha,\beta$ of the same size. Therefore, combining \eqref{subsolntest0} and \eqref{subsolntest1}, we have
\begin{align}
0\geq K+\frac{\hat{t}-\hat{s}}{\varepsilon^{1/2}}+\overline{F}\left(\frac{\hat{z}-\hat{y}}{\varepsilon^{1/2}}\right)-C(\varepsilon^{1/2}+\gamma).\label{subsolntest}
\end{align}

Now, we fix $(\hat{x},\hat{z},\hat{t})$. For $\sigma>0$, we let
$$
\Psi(y,\xi,s):=u(y,s)+\varepsilon v\left(\frac{\hat{x}}{\varepsilon},\frac{\hat{z}-\xi}{\varepsilon^{1/2}}\right)+\frac{|\hat{x}-y|^2+|\hat{t}-s|^2}{2\varepsilon^{1/2}}+\frac{|y-\xi|^2}{2\sigma}+Ks.
$$
Then, $\Psi$ attains a minimum at $(y_{\sigma},\xi_{\sigma},s_{\sigma})\in\mathbb{R}^{2n}\times[0,T]$, and $(y_{\sigma},\xi_{\sigma},s_{\sigma})\to(\hat{y},\hat{y},\hat{t})$ as $\sigma\to0$ upto a subsequence. From $\Psi(y_{\sigma},\xi_{\sigma},s_{\sigma})\leq\Psi(y_{\sigma},y_{\sigma},s_{\sigma})$, we get
$$
\frac{|y_{\sigma}-\xi_{\sigma}|^2}{2\sigma}\leq\varepsilon\left(v\left(\frac{\hat{x}}{\varepsilon},\frac{\hat{z}-y_{\sigma}}{\varepsilon^{1/2}}\right)-v\left(\frac{\hat{x}}{\varepsilon},\frac{\hat{z}-\xi_{\sigma}}{\varepsilon^{1/2}}\right)\right)\leq C\varepsilon^{1/2}|y_{\sigma}-\xi_{\sigma}|,
$$
which implies $|y_{\sigma}-\xi_{\sigma}|\leq C\varepsilon^{1/2}\sigma$. Now, $(y,s)\mapsto\Psi(y,\xi_{\sigma},s)$ attains a minimum at $(y_{\sigma},s_{\sigma})$, and by the viscosity supersolution test for $u$ at $(y_{\sigma},s_{\sigma})$, we obtain
$$
-K-\frac{s_{\sigma}-\hat{t}}{\varepsilon^{1/2}}+\overline{F}\left(-\frac{y_{\sigma}-\hat{x}}{\varepsilon^{1/2}}-\frac{y_{\sigma}-\xi_{\sigma}}{\sigma}\right)\geq0.
$$
Letting $\sigma\to0$, we get
\begin{align}
-K+\frac{\hat{t}-\hat{s}}{\varepsilon^{1/2}}+\overline{F}\left(\frac{\hat{x}-\hat{y}}{\varepsilon^{1/2}}\right)\geq -C\varepsilon^{1/2}.\label{supersolntest}
\end{align}

Combining \eqref{subsolntest} and \eqref{supersolntest}, we obtain
$$
2K\leq C\varepsilon^{1/2}+C\gamma.
$$
For the choices $\gamma\leq\varepsilon^{1/2},\ K=K_1\varepsilon^{1/2}$, $K_1>0$ a constant sufficiently large, we see that this is a contradiction.

Therefore, we have either $\hat{t}=0$ or $\hat{s}=0$. In case when $\hat{s}=0$, we have $\hat{t}\leq C\varepsilon^{1/2}$, and therefore, we obtain $u^{\varepsilon}(\hat{x},\hat{t})-u_0(\hat{x})=\int_0^{\hat{t}}u_t^{\varepsilon}(\hat{x},s)ds\leq C\varepsilon^{1/2}$ by using \eqref{line:timederivative}. Consequently,
$$
\Phi(\hat{x},\hat{y},\hat{z},\hat{t},\hat{s})\leq u^{\varepsilon}(\hat{x},\hat{t})-u(\hat{y},\hat{s})-\varepsilon v\left(\frac{\hat{x}}{\varepsilon},\frac{\hat{z}-\hat{y}}{\varepsilon^{1/2}}\right)\leq C\varepsilon^{1/2}.
$$
In case when $\hat{t}=0$, the above follows from the fact that $\|u_t\|_{L^{\infty}(\mathbb{R}^n\times[0,\infty))}\leq C$.

Since $\Phi(x,x,x,t,t)\leq C\varepsilon^{1/2}$ for all $(x,t)\in\mathbb{R}^n\times[0,T]$, it holds that
$$
u^{\varepsilon}(x,t)-u(x,t)\leq C\varepsilon^{1/2}+\varepsilon v\left(\frac{x}{\varepsilon},0\right)+2K_1\varepsilon^{1/2}t+\gamma\langle x\rangle.
$$
By letting $\gamma\to0$, we obtain the upper bound
$$
u^{\varepsilon}(x,t)-u(x,t)\leq C(1+T)\varepsilon^{1/2}
$$
for all $(x,t)\in\mathbb{R}^n\times[0,T]$.

To prove the lower bound
$$
u^{\varepsilon}(x,t)-u(x,t)\geq -C(1+T)\varepsilon^{1/2}
$$
for all $(x,t)\in\mathbb{R}^n\times[0,T]$, we alternatively consider another auxiliary function
\begin{align*}
\Phi_1(x,y,z,t,s):=u^{\varepsilon}(x,t)&-u(y,s)-\varepsilon v\left(\frac{x}{\varepsilon},\frac{z-y}{\varepsilon^{1/2}}\right)\\
&+\frac{|x-y|^2+|t-s|^2}{2\varepsilon^{1/2}}+\frac{|x-z|^2}{2\varepsilon^{1/2}}+K(t+s)+\gamma\langle x\rangle.
\end{align*}
Then, we follow a similar argument as the above to obtain the lower bound.
\end{proof}


\section{Proof of Theorem 1.2}\label{sec:example}
In this section, we prove Theorem \ref{thm:example}. We let $c(x)=0$, $u_0(x)=|x|$ for $x\in\mathbb{R}^n$. Note that $u(x,t)=|x|$ is the unique viscosity solution to the effective equation \eqref{eq:effec} with $\overline{F}\equiv0$.

\begin{proof}[Proof of Theorem \ref{thm:example}]
Let $\overline{w}=\overline{w}(x,t)$ be the unique Lipschitz classical solution to the mean curvature flow
\begin{equation}\label{eq:csf}
\begin{cases}
\overline{w}_t=\textrm{tr}\left\{a(D\overline{w})D^2\overline{w}\right\} \quad &\text{ in } \mathbb{R}^n\times(0,\infty),\\
\overline{w}(x,0)=|x| \quad &\text{ on } \mathbb{R}^n.
\end{cases}
\end{equation}
As \eqref{eq:csf} enjoys the comparison principle among Lipschitz solutions, we have that $\frac{1}{\lambda}\overline{w}(\lambda x,\lambda^2 t)=\overline{w}(x,t)$ for any $\lambda>0$. Therefore, $\overline{w}(0,t)=\sqrt{t}\overline{w}(0,1)$ for any $t\geq0$, and the fact that $\overline{w}(0,1)>0$ can be proved by taking a barrier function from below whose initial data is a smooth convex function that passes through the origin and is less or equal to the function $w_0(x)=|x|$.

Since $u^{\varepsilon}(x,t)=\varepsilon \overline{w}\left(\frac{x}{\varepsilon},\frac{t}{\varepsilon}\right)$ and $u(x,t)=|x|$, we have
\begin{align*}
&u^{\varepsilon}(0,t)-u(0,t)
=\varepsilon \overline{w}\left(0,\frac{t}{\varepsilon}\right)=\overline{w}(0,1)\sqrt{t\varepsilon}>0.
\end{align*}
By taking $t=1,\ C=\overline{w}(0,1)>0$, we complete the proof of Theorem \ref{thm:example}.
\end{proof}

\appendix
\section{Proof of Theorem \ref{thm:wellposedness}}\label{appendix}
We prove Theorem \ref{thm:wellposedness} in this appendix. We will separate the steps into Propositions \ref{prop:shorttime}, \ref{prop:timederivative}, \ref{prop:gradientestimate}, whose statements are about the estimates of gradients, Hessians and time derivatives.

We state the short-time existence of classical solutions to \eqref{eq:w}. We skip the proof as known in the literature. The uniqueness follows from the standard comparison principle, for which we refer to \cite{CIL}. See \cite{BBBL} for more general results in this direction. For the existence with gradient and Hessian estimates, we refer to \cite{EH} (in the absence of a forcing) and to \cite[Appendix A]{DKY} (when with a $C^2$ forcing term).

\begin{proposition}\label{prop:shorttime}
Let $w_0$ be a globally Lipschitz function on $\mathbb{R}^n$ with $\|Dw_0\|_{L^{\infty}(\mathbb{R}^n)}$ $\leq N_0<+\infty$. Then, there exists $T^*=T^*(\|c\|_{C^2(\mathbb{R}^n)},N_0)>0$ such that \eqref{eq:w} with $T=T^*$ has a unique classical solution $w=w(x,t)$. Moreover, for each $T\in(0,T^*)$, there exist $N=N(\|c\|_{C^2(\mathbb{R}^n)},N_0,T)>0$ and $C=C(\|c\|_{C^2(\mathbb{R}^n)},N_0,T)$ $>0$ such that
$$
\|Dw\|_{L^{\infty}(\mathbb{R}^n\times[0,T])}\leq N,
$$
and
$$
\|D^2w(\cdot,t)\|_{L^{\infty}(\mathbb{R}^n)}\leq \frac{C}{\sqrt{t}}
$$
for $t\in(0,T]$.
\end{proposition}

Next, we state and prove a priori time derivative estimates based on the maximum principle. See \cite[Lemma 3.1]{JKMT}.

\begin{proposition}\label{prop:timederivative}
Let $w_0$ be a globally Lipschitz function on $\mathbb{R}^n$ with $\|Dw_0\|_{L^{\infty}(\mathbb{R}^n)}$ $\leq N_0<+\infty$. Let $T^*=T^*(\|c\|_{C^2(\mathbb{R}^n)},N_0)>0$ be chosen such that \eqref{eq:w} with $T=T^*$ has the unique classical solution $w=w(x,t)$. Then, for any $\tau\in(0,T^*)$, it holds that
$$
\|w_t\|_{L^{\infty}(\mathbb{R}^n\times[\tau,T^*))}\leq\|w_t(\cdot,\tau)\|_{L^{\infty}(\mathbb{R}^n)}<+\infty.
$$
\end{proposition}

\begin{proof}[Proof of Proposition \ref{prop:timederivative}]
Let $T\in(\tau,T^*)$. Then, by Proposition \ref{prop:shorttime}, there exist $N=N(\|c\|_{C^2(\mathbb{R}^n)},N_0,T)>0$ and $C=C(\|c\|_{C^2(\mathbb{R}^n)},N_0,T)>0$ such that
$$
\|Dw\|_{L^{\infty}(\mathbb{R}^n\times[0,T])}\leq N,
$$
and
$$
\|D^2w(\cdot,t)\|_{L^{\infty}(\mathbb{R}^n)}\leq \frac{C}{\sqrt{t}}
$$
for $t\in(\tau,T]$. Therefore, by the equation \eqref{eq:w}, we see that $\|w_t\|_{L^{\infty}(\mathbb{R}^n\times[\tau,T])}\leq K=K(\|c\|_{C^2(\mathbb{R}^n)},N_0,\tau,T)$.

We aim to prove
$$
\sup_{\mathbb{R}^n\times[\tau,T]}w_t\leq\sup_{\mathbb{R}^n}w_t(\cdot,\tau).
$$
Suppose for the contrary that there exists $(x_0,t_0)\in\mathbb{R}^n\times(\tau,T]$ such that
$$
w_t(x_0,t_0)>\sup_{\mathbb{R}^n}w_t(\cdot,\tau).
$$
Then there would exist a number $\lambda\in(0,1)$ such that
$$
w_t(x_0,t_0)-\lambda t_0>\sup_{\mathbb{R}^n}w_t(\cdot,\tau)-\lambda \tau.
$$

We run Bernstein method now with $\Phi(x,t):=w_t(x,t)-\lambda t$. Let $\Phi^*(t):=\sup_{\mathbb{R}^n}\Phi(\cdot,t)$ for each $t\in[\tau,T]$. Then $\Phi^*(t_0)>\Phi^*(\tau)$. Fix a sequence $\{\varepsilon_j\}_j$ of positive numbers that converges to $0$ as $j\to\infty$. For each $t\in[\tau,T]$, let $x_j(t)$ be a maximizer of $\Phi_j(x,t)=\Phi(x,t)-\varepsilon_j|x|^2$. Then, $\Phi(x_j(t),t)\to\Phi^*(t),\ D\Phi(x_j(t),t)\to0$ as $j\to\infty$, and $\limsup_{j\to\infty}D^2\Phi(x_j(t),t)\leq0$ in the sense that $\limsup_{j\to\infty}(D^2\Phi(x_j(t),t)v)\cdot v\leq0$ for any $v\in\mathbb{R}^n$.

Note that $\{t\in[\tau,T]:\Phi^*(t)=\sup_{[\tau,T]}\Phi^*(\cdot)\}$ is a closed subinterval of $[\tau,T]$ not containing $\tau$. Consequently, there exists $t^*\in(\tau,T]$ such that $\Phi^*(t^*)=\sup_{[\tau,T]}\Phi^*(\cdot),$ $\Phi^*(t)<\Phi^*(t^*)$ for all $t\in[\tau,t^*)$, and thus that $\liminf_{j\to\infty}\Phi_t(x_j(t^*),t^*)\geq0$.  

Differentiating the first line of \eqref{eq:w} in $t$, we obtain
$$
(w_t)_t-\textrm{tr}\{a(Dw)D^2w_t\}=\textrm{tr}\{a(Dw)_tD^2w\}+c\frac{Dw\cdot Dw_t}{\sqrt{1+|Dw|^2}}.
$$
Also, at $(x,t)\in\mathbb{R}^n\times[\tau,T]$,
\begin{align*}
\textrm{tr}\{a(Dw)_tD^2w\}&=\textrm{tr}\{(D_pa(Dw)\odot Dw_t)D^2w\}\\
&\leq\frac{4n^3}{\sqrt{1+|Dw|^2}}|Dw_t|\|D^2w\|\leq C(n,\|c\|_{C^2(\mathbb{R}^n)},N_0,\tau)|Dw_t|
\end{align*}
for some constant $C=C(n,\|c\|_{C^2(\mathbb{R}^n)},N_0,\tau)>0$ depending only on its argument. Here, we have used the fact that $\left|\frac{\partial}{\partial p_k}a^{ij}(p)\right|\leq\frac{4}{\sqrt{1+|p|^2}}$ for all $p\in\mathbb{R}^n$. Also,
$$
c\frac{Dw\cdot Dw_t}{\sqrt{1+|Dw|^2}}\leq\|c\|_{L^{\infty}(\mathbb{R}^n)}|Dw_t|.
$$
Therefore, evaluated at $(x_j(t^*),t^*)$ in the following limit,
\begin{align*}
0&\leq\liminf_{j\to\infty}\left(\Phi_t-\textrm{tr}\left\{a(Dw)D^2\Phi\right\}\right)\\
&\leq\liminf_{j\to\infty}\left(-\lambda+(w_t)_t-\textrm{tr}\{a(Dw)D^2w_t\}\right)\\
&\leq-\lambda+\liminf_{j\to\infty}C(n,\|c\|_{C^2(\mathbb{R}^n)},N_0,\tau)|Dw_t|\\
&=-\lambda,
\end{align*}
a contradiction.

The statement $\inf_{\mathbb{R}^n\times[\tau,T]}w_t\geq\inf_{\mathbb{R}^n}w_t(\cdot,\tau)$ can be verified similarly, and $T\in(\tau,T^*)$ can be chosen arbitrarily. Therefore, we complete the proof.
\end{proof}

We state and prove a priori gradient estimates. The point of the following proposition is to remove the dependency on $T\in(0,T^*)$ in the estimate of Proposition \ref{prop:shorttime}. We refer to \cite{LS,JKMT} regarding gradient estimates from the coercivity condition (A3).

\begin{proposition}\label{prop:gradientestimate}
Let $w_0$ be a globally Lipschitz function on $\mathbb{R}^n$ with $\|Dw_0\|_{L^{\infty}(\mathbb{R}^n)}$ $\leq N_0<+\infty$. Let $T^*=T^*(\|c\|_{C^2(\mathbb{R}^n)},N_0)>0$ be chosen such that \eqref{eq:w} with $T=T^*$ has the unique classical solution $w=w(x,t)$. Then, for any $\tau\in(0,T^*)$, there exists $M=M(n,\|c\|_{L^{\infty}(\mathbb{R}^n)},\|w_t(\cdot,\tau)\|_{L^{\infty}(\mathbb{R}^n)},\delta)>0$
such that
$$
\|Dw\|_{L^{\infty}(\mathbb{R}^n\times[\tau,T^*))}\leq\max\{\|Dw(\cdot,\tau)\|_{L^{\infty}(\mathbb{R}^n)},M\}.
$$
Here, $\delta>0$ is the number appearing in the condition \rm(A3).
\end{proposition}

\begin{proof}[Proof of Proposition \ref{prop:gradientestimate}]
Let $T\in(\tau,T^*)$. By Proposition \ref{prop:shorttime}, there exists $N=N(\|c\|_{C^2(\mathbb{R}^n)},N_0,T)>0$ such that
\begin{align}\label{gradientestimateintime}
\|Dw\|_{L^{\infty}(\mathbb{R}^n\times[\tau,T])}\leq N. 
\end{align}
The goal of this proof is to make this estimate independent of $T\in(\tau,T^*)$.

We now run Bernstein method with $\Phi(x,t):=z(x,t)$. Let $\Phi^*(t):=\sup_{\mathbb{R}^n}w(\cdot,t)$ for $t\in[\tau,T]$. Let $\{\varepsilon_j\}_j$ be a sequence of positive numbers that converges to $0$ as $j\to\infty$. For each $t\in[\tau,T]$, a maximizer $\{x_j(t)\}_j$ of $\Phi_j(x,t):=\Phi(x,t)-\varepsilon_j|x|^2$ satisfies that $\Phi(x_j(t),t)\to\Phi^*(t),\ D\Phi(x_j(t),t)\to0$ as $j\to\infty$, and that $\limsup_{j\to\infty}D^2\Phi(x_j(t),t)\leq0$. Here, we are using the estimate \eqref{gradientestimateintime}.

If $\{t\in[\tau,T]:\Phi^*(t)=\sup_{[\tau,T]}\Phi^*(\cdot)\}$ contains $\tau$, we obtain the conclusion. We assume the other case so that there exists $t_1\in(\tau,T]$ such that $\Phi^*(t)<\Phi^*(t_1)=\sup_{[\tau,T]}\Phi^*(\cdot)$ for all $t\in[\tau,t_1)$. Then, it holds that $\liminf_{j\to\infty}\Phi_t(x_j(t_1),t_1)\geq0$.

We differentiate the first line of \eqref{eq:w} in $x_k$ and multiply by $w_{x_k}$ and then sum over $k=1,\cdots,n.$ We get, as a result,
\begin{align}
zz_t-z\textrm{tr}\{a(Dw)D^2z\}=z\textrm{tr}\{(D_pa(Dw)\odot Dz)D^2w\}-\textrm{tr}\{(a(Dw)D^2w)^2\}\nonumber\\
\qquad\qquad\qquad\qquad\qquad\qquad\qquad+zDc\cdot Dw+cDw\cdot Dz.\label{linearization}
\end{align}

We estimate the term $\textrm{tr}\{(a(Dw)D^2w)^2\}$. Using the fact that $Dz=z^{-1}D^2wDw$, we see that
\begin{align}
&\textrm{tr}\{(a(Dw)D^2w)^2\}\nonumber\\
&=\textrm{tr}\{a(Dw)D^2wI_nD^2w\}-\textrm{tr}\{a(Dw)D^2w\frac{Dw}{z}\otimes\frac{Dw}{z}D^2w\}\nonumber\\
&=\textrm{tr}\{a(Dw)(D^2w)^2\}-\textrm{tr}\{a(Dw)Dz\otimes Dz\}.\label{mainterm1}
\end{align}
Recall Cauchy-Schwarz's inequality; for two square matrices $\alpha,\beta$ of the same size, we have
$$
\textrm{tr}\{\alpha\beta^t\}^2\leq\|\alpha\|^2\|\beta\|^2
$$
Assume $n\geq2$. We put $\alpha=(a(Dw))^{1/2}D^2w,\ \beta=(a(Dw))^{1/2}$ to obtain
\begin{align}
&\textrm{tr}\{a(Dw)(D^2w)^2\}\nonumber\\
&\geq\frac{1}{n-1+\frac{1}{z^2}}\textrm{tr}\{a(Dw)D^2w\}^2\nonumber\\
&\geq\frac{1}{n-1}(w_t-cz)^2-\frac{1}{(n-1)^2z^2}(w_t-cz)^2\nonumber\\
&\geq\frac{c^2}{n-1}z^2-\frac{2\|w_t(\cdot,\tau)\|_{L^{\infty}(\mathbb{R}^n)}\|c\|_{L^{\infty}(\mathbb{R}^n)}}{n-1}z-C\label{mainterm2}
\end{align}
for some constant $C=C(n,\|c\|_{L^{\infty}(\mathbb{R}^n)},\|w_t(\cdot,\tau)\|_{L^{\infty}(\mathbb{R}^n)})>0$. Here, we have used Proposition \ref{prop:timederivative}.

Note that
$$
z\textrm{tr}\{(D_pa(Dw)\odot Dz)D^2w\}\leq 4n^3|Dz|\|D^2w\|_{L^{\infty}(\mathbb{R}^n\times[\tau,T])}\leq C|Dz|
$$
for some constant $C=C(n,\|c\|_{C^2(\mathbb{R}^n)},N_0,\tau,T)>0$ from the Hessian estimate in Proposition \ref{prop:shorttime}. We also have used the fact that $\left|\frac{\partial}{\partial p^k}a^{ij}(p)\right|\leq\frac{4}{\sqrt{1+|p|^2}}$ for $p\in\mathbb{R}^n$.

From \eqref{linearization}, \eqref{mainterm1}, \eqref{mainterm2}, we have
\begin{align*}
&zz_t-z\textrm{tr}\{a(Dw)D^2w\}\\
&\leq C(n,\|c\|_{C^2(\mathbb{R}^n)},N_0,\tau,T)|Dz|-\left(\frac{c^2}{n-1}-|Dc|\right)z^2\\
&\ \ \ +\frac{2\|w_t(\cdot,\tau)\|_{L^{\infty}(\mathbb{R}^n)}\|c\|_{L^{\infty}(\mathbb{R}^n)}}{n-1}z\ +|Dz|^2+\|c\|_{L^{\infty}(\mathbb{R}^n)}|Dz|z+C
\end{align*}
for some constant $C=C(n,\|c\|_{L^{\infty}(\mathbb{R}^n)},\|w_t(\cdot,\tau)\|_{L^{\infty}(\mathbb{R}^n)})>0$. Evaluate at $(x_j(t_1),t_1)$ and let $j\to\infty$ to obtain
\begin{align*}
0\leq-\delta \Phi^*(t_1)^2+C\Phi^*(t_1)+C
\end{align*}
for some constant $C=C(n,\|c\|_{L^{\infty}(\mathbb{R}^n)},\|w_t(\cdot,\tau)\|_{L^{\infty}(\mathbb{R}^n)})>0$ (taking a larger one than the previous lines if necessary) depending only on its arguments. This completes the proof when $n\geq2$. In the case of $n=1$, the estimate can be carried out similarly.
\end{proof}

Combining Propositions \ref{prop:shorttime}, \ref{prop:timederivative}, \ref{prop:gradientestimate}, we obtain the long-time existence, proved in the following.

\begin{proof}[Proof of Theorem \ref{thm:wellposedness}]
The uniqueness is standard \cite{CIL,BBBL}. We prove the existence of classical solutions for all time.

Let $T^*=T^*(\|c\|_{C^2(\mathbb{R}^n)},N_0)>0$ be chosen as in Proposition \ref{prop:shorttime}. Fix $\tau_0=\frac{1}{2}T^*$. Tracking the dependency on parameters using Propositions \ref{prop:shorttime}, \ref{prop:timederivative}, \ref{prop:gradientestimate}, we see that there exists $M=M(n,\|c\|_{C^2(\mathbb{R}^n)},N_0,\delta)>0$ such that
$$
\|Dw\|_{L^{\infty}(\mathbb{R}^n\times[\tau_0,T^*))}\leq M.
$$
Let $\tau_1=T^*-\varepsilon\in(\tau_0,T^*)$. Starting from $w(\cdot,\tau_1)$ at $t=\tau_1$, seen as an initial data, we can extend the solution on time interval $[\tau_1,\tau_1+\frac{1}{(C_1+1)\sqrt{1+M^2}})$ (see the proof of Proposition \ref{prop:shorttime} for this explicit expression). As $\varepsilon>0$ can be arbitrarily small, the solution exists on time interval $[0,T_1^*)$ with $T_1^*=T^*+\frac{1}{(C_1+1)\sqrt{1+M^2}}$. Not changing the choice $\tau_0=\frac{1}{2}T^*$, we still have
$$
\|Dw\|_{L^{\infty}(\mathbb{R}^n\times[\tau_0,T_1^*))}\leq M.
$$
with the same constant $M=M(n,\|c\|_{C^2(\mathbb{R}^n)},N_0,\delta)>0$ by applying the proofs of Propositions \ref{prop:timederivative}, \ref{prop:gradientestimate}. Then, we can extend the solution on time interval $[0,T_2^*)$ with $T_2^*=T^*+\frac{2}{(C_1+1)\sqrt{1+M^2}}$ as we just did from $[0,T^*)$ to $[0,T_1^*)$. We inductively proceed to conclude the solution exists for all time.

The estimate \eqref{line:timederivativeofw} is a simple consequence of Propositions \ref{prop:shorttime}, \ref{prop:timederivative}.
\end{proof}

\section*{Acknowledgements}
The author would like to thank Olivier Ley, Qing Liu, Hiroyoshi Mitake, Norbert Pozar, Hung V. Tran, Yifeng Yu for helpful discussions and comments.

\end{document}